\documentclass[letterpaper,11pt]{amsart}
\usepackage{etex}
\usepackage{ifdraft}

\usepackage[utf8]{inputenc}
\usepackage[T1]{fontenc}

\usepackage{lmodern}
\usepackage{fourier}
\usepackage{amssymb, amsfonts}

\usepackage{mathrsfs} % script characters (\mathscr)
\usepackage{dsfont}
\usepackage[usenames,dvipsnames]{xcolor}%, colortbl}
\usepackage{xcolor}%, colortbl}
\usepackage{url}

\usepackage{stmaryrd}
\usepackage{mathrsfs}
\usepackage[english]{babel}

\usepackage[margin=7em]{geometry}
\usepackage{needspace}

\usepackage{amsmath}
\ifdraft{{}}
  {\usepackage[pdftex,
               pdfborder={0 0 0}, colorlinks=true,
               linkcolor=BrickRed, citecolor=ForestGreen, urlcolor=RoyalBlue]{hyperref}}

\ifdraft{\usepackage[textsize=tiny,colorinlistoftodos]{todonotes}}{}

\usepackage{booktabs}
\usepackage[inline]{enumitem}
\usepackage{float}

\usepackage[all]{xy}

\usepackage{tikz}
\usepackage{tikz-cd}
\usetikzlibrary{matrix,arrows,decorations.pathmorphing}

\usepackage{placeins}%to clear floats use command \FloatBarrier.
\usepackage{mathtools}
\usepackage{booktabs, array}

\usepackage[format=hang]{caption}
\usepackage{pifont}% http://ctan.org/pkg/pifont
\newcommand{\cmark}{\ding{51}}%
\newcommand{\xmark}{\ding{55}}%
\usepackage{color}
\usepackage{graphicx} %, nicefracm, 
\usepackage{amsmath,amssymb, verbatim}
\usepackage{latexsym}
\usepackage{mathtools}
\theoremstyle{plain}
   
   \newtheorem{theorem}{Theorem}[section]
   
   \newtheorem{prop}[theorem]{Proposition}
   \newtheorem{lemma}[theorem]{Lemma}

\theoremstyle{definition}
   \newtheorem{definition}[theorem]{Definition}
   \newtheorem{example}[theorem]{Example}

   \newtheorem{remark}[theorem]{Remark}

\numberwithin{equation}{section}
\newenvironment{smallpmatrix}[1][]{\left(\begin{smallmatrix}#1}{\end{smallmatrix}\right)}

\newcommand{\RR}{{\mathbb {R}}}
\newcommand{\ZZ}{{\mathbb {Z}}}

\newcommand{\om}{\omega}

\newcommand\CC{{\mathbb C}}

\newcommand\BBZ{{\mathbb Z}}

\newcommand\GL{{\operatorname{GL}}}

\DeclareMathOperator*{\Tran}{Tran}
\DeclareMathOperator*{\Lin}{Lin}

\newcommand\diag[1]{{\operatorname{diag}(#1)}}
\begin{document}
\title[Steinberg's Theorem]{Steinberg's theorem \\ for
  crystallographic complex
reflection groups}

\author{Philip\ Puente}
\author{Anne V.\ Shepler}
\thanks{Research supported in part by NSF Grant DMS-1101177
and Simons Foundation Grant \#429539.}
\email{philip.c.puente@dartmouth.edu, ashepler@unt.edu}
\address{Department of Mathematics\\Dartmouth College\\Hanover, NH 03755, USA}
\address{Department of Mathematics\\University of North Texas\\Denton, TX 76203, USA}

\date{February 23, 2018.}


\begin{abstract}
Popov classified crystallographic
complex reflection groups by determining lattices they
stabilize.
These analogs of affine Weyl groups have infinite order 
and are generated by reflections about
affine hyperplanes; most arise
as the semi-direct product of a finite complex reflection group and a full rank lattice.
Steinberg's fixed point theorem asserts that the
regular orbits under the action of a reflection group
are exactly the orbits lying off of reflecting hyperplanes.
This theorem holds for finite reflection groups (real or complex)
and also affine Weyl groups but fails for some crystallographic
complex reflection groups.
We determine when Steinberg's theorem holds for the infinite family
of crystallographic complex reflection groups.
We include crystallographic groups built on finite Coxeter groups.
\end{abstract}

\maketitle

\vspace{-3ex}

\section{Introduction}
We investigate Steinberg's fixed point theorem for affine complex reflection groups,
analogs of affine Weyl groups 
called {\em crystallographic complex reflection groups}.
These discrete groups were classified by Popov~\cite{Popov} in 1982 and
are generated by reflections about affine hyperplanes, i.e., mirrors
that do not necessarily include the origin.
They combine 
the isometries of a finite
complex reflection group $G\leq\text{GL}(V)$ with 
translation along a $G$-invariant lattice $\Lambda$ in $V=\CC^n$,
but they can not always be described as a semi-direct product $G\ltimes \Lambda$.
The adjective {\em crystallographic} indicates
that the orbit space $V/W$
for the action of $W$ on $V$ is compact.

A point under
the action of a group is
called {\em regular} if its 
stabilizer in the group is trivial.
Steinberg~\cite{Steinberg, Steinberg68} showed that the regular points under the action of a finite
reflection group (real or complex) or an affine Weyl group are precisely those lying off of the reflecting hyperplanes.
A famous (or infamous)
series of exercises in
Bourbaki~\cite[Ch.~V, \S5, Ex.~8]{Bourbaki}
also outlines a proof of this
fact using Auslander~\cite{Auslander}. 
More recently,
Lehrer~\cite{Lehrer} gave a proof based on elementary
invariant theory.
Steinberg~\cite{Steinberg68} remarked 
that this result
is sometimes known as
"Chevalley's Theorem" in the case of
Coxeter groups although
it was known to Cartan and Weyl.
We show that Steinberg's theorem holds for 
most crystallographic complex
reflection groups but not all.
In fact, it fails for some groups
whose underlying
linear parts are finite Coxeter groups.

Every affine complex reflection group is the direct sum of irreducible ones (or trivial groups),
and the conclusion of Steinberg's theorem
is preserved under
direct sum.
Thus one asks:
For which
irreducible groups
does
Steinberg's 
theorem hold?
Popov's classification~\cite{Popov} 
of irreducible crystallographic 
reflection groups
comprises one infinite family of groups
$[G(r,p,n)]_k=G(r,p,n)\ltimes
\Lambda$
depending on $4$ parameters
and some exceptional groups.
The $4$-parameter
family combines the 
$3$-parameter family of finite complex reflection groups $G(r,p,n)$
(which includes the infinite
families of Coxeter groups)
with various invariant lattices
$\Lambda$ in $\CC^n$.
We classify those groups in this 
family for which non-regular orbits
all lie on reflecting hyperplanes, thus determining when Steinberg's theorem
holds.  We treat the case
of crystallographic groups built on Coxeter
groups separately.

\begin{theorem}
   Let $\Lambda$ be a $G(r,p,n)$-invariant
  lattice of full rank $2n$ in $\CC^n$ for $r,p,n\geq 1$.  Assume $G(r,p,n)$
  is not a Coxeter group.
  The set of nonregular points
  for  $W=G(r,p,n)\ltimes \Lambda$
acting on $\CC^n$
is the union of reflecting
affine 
hyperplanes for $W$
if and only if\/ $r \neq p$ and
$W\neq[G(3,1,n)]_2$ and $W\neq[G(6,3,2)]_2$.
\end{theorem}

We also determine those
crystallographic reflection groups
$G\ltimes \Lambda$
for which Steinberg's theorem fails when $G$ lies in an infinite family of
Weyl groups, see~Theorem~\ref{nongenuine}:
The theorem fails
except when $G$
is $W(A_{n-1})$
or 
$W(B_n)$,
and even then it does not always hold.
Our arguments rely on 
analysis of orbits under various reflection groups of infinite order acting on subsets of $\CC$.
The next two examples
are those mentioned in the theorem.
We use the standard basis
$e_1,\ldots, e_n$ of $V=\CC^n$.

%%%%%%%%%%%%%%%%%%%%%%%%
\begin{example}
The group $[G(6,3,2)]_2$ acting on $\CC^2$ is $G(6,3,2)\ltimes \Lambda$
for $\ZZ$-lattice in $\CC^2$ of rank $4$
$$
\Lambda=\ZZ[2\xi](\xi e_1-e_2)+ \ZZ[2\xi](1-\xi)(e_1-e_2),
$$
where $\xi=e^{2 \pi i/6}$
and where
$G(6,3,2)$ is the linear group generated
by
$
\begin{smallpmatrix}
0& 1\rule[-.5ex]{0ex}{1.5ex}\\
1 & 0
\end{smallpmatrix},
\begin{smallpmatrix}
\xi^3 & 0\\
0 & 1
\end{smallpmatrix},
\text{ and }
\begin{smallpmatrix}
0 & \xi \\
\xi^{-1} & 0
\end{smallpmatrix}.
$
\end{example}
%%%%%%%%%%%%%%%%%%%%%%%
\begin{example}
The group $[G(3,1,n)]_2$ acting on $\CC^n$
is $W=G(3,1,n)\ltimes \Lambda$ for
$\ZZ$-lattice of rank $2n$
$$
\begin{aligned}
\Lambda\ &=\
\ZZ[\om]e_1\oplus\Sigma_{k=2}^{n}\ \ \ZZ[\om]\, \left(\tfrac{1}{1-\om}\right)(e_{k-1}-e_k)\\
\end{aligned}
$$
where $\om=e^{2 \pi i/3}$.
Here,
$G(3,1,n)$ is the finite complex
reflection group
generated by the $n\times n$ permutation
matrices together with the diagonal matrix
$\diag{\om, 1, \ldots, 1}$.
\end{example}

\subsection*{Outline}
We give basic notions 
and recall
Steinberg's fixed point theorem for
reflection groups in Section~\ref{setup}.  We review
Popov's classification in Section~\ref{classification}.
The case of crystallographic groups built upon the symmetric group $\mathfrak{S}_n$ appears
in Section~\ref{symmetricgroupsection}.  We determine the genuine groups $G(r,p,n)\ltimes \Lambda$
satisfying Steinberg's theorem in
Section~\ref{groupspassing}
and those for which the theorem fails in Section~\ref{groupsfailing}.  In Section~\ref{Coxetersection}, we consider crystallographic groups built upon Coxeter groups.

%%%%%%%%%%%%%%%%%%%%%%%%%%%%%%%%%%%%%%%%%%%55
%%%%%%%%%%%%%%%%%%%%%%%%%%%%%%%%%%%%%%%%%%%%%555

%%%%%%%%%%%%%%%%%%%%%%%%%%%%%%%%%%%%%
\section{Crystallographic reflection groups}
\label{setup}

We fix a positive definite inner product on $V=\CC^n$
and standard basis
$e_1,\ldots, e_n$ of $V$
with dual basis $x_1,\ldots, x_n$
of $V^*$.
All lattices are $\ZZ$-lattices
and 
a lattice in $\CC^n$ 
has {\em full rank} if
it has rank $2n$.

%%%%%%%%%%%%%%%%%%%%%%%%%%%%%%%%%%%%%%%%%%%%%%%%%%%%%%%%%%%%%
\subsection*{Affine transformations}
We identify the set of
affine transformations on $V$ with $\text{M}_n(\CC)\ltimes V$
so that 
an {\em affine transformation} $g$
of $V$ is the composition
of a linear transformation and a translation:
$$g(v)=\Lin(g)(v)+\Tran(g)
\quad\text{ for } v \text{ in }V\, ,
$$ 
for some fixed matrix
$\Lin(v)\in M_n(\CC)$,
the {\em linear part} of $g$, and a fixed vector $\Tran(g)=g(0)\in V$,
the {\em translational part} of $g$.
Note that $\Lin(g)$
is the map $v\mapsto g(v)-g(0)$.
The set of invertible
affine transformations $A(V)$ is
identified with
$GL(V)\ltimes V$ via
$g\mapsto (\Lin(g), g(0))$
and we have a map
$\Lin:A(V)\rightarrow \GL(V)$.

%%%%%%%%%%%%%%%%%%%%%%%%%%%%%%%%%%%%%%%%%%5
\subsection*{Affine reflections}
An {\em affine reflection} 
(or just {\em reflection}) on $V$ is a non-identity affine isometry $s$
fixing an affine hyperplane $H_s$ in $\CC^n$ pointwise,
called the {\em reflecting hyperplane} of $s$. %
A reflection $s$ is {\em central} if $s(0)=0$ or, equivalently, 
if $H_s=\ker(s-1_V)$ is a linear subspace of $V$ and $s$ is a linear
transformation.
When $s$ is a central reflection of finite order, there is a vector $\alpha_s \perp H_s$
with $s(\alpha_s)=\xi\alpha_s$ for some primitive $m$-th root-of-unity
$\xi$ in $\CC$, the non-identity eigenvalue of $s$,
where $m$ is the order of $s$.  

The lemma below
shows that every affine reflection 
is obtained by composing
a central reflection
with translation by a vector perpendicular to the 
central reflecting hyperplane
(compare with Popov~\cite[Subsection 1.2]{Popov}).
We give a proof
of this fact and related observations we need later for completeness.

%%%%%%%%%%%%%%%%%%%%%%%%%%%%%%%%%
\begin{lemma}
\label{geometrylemma}
Suppose $g$ is an affine transformation on $V$
and $\Lin(g)$
has finite order.
Then
\begin{itemize}
\item[(1)] 
The transformation $g$ has finite order
if and only if\/ $g$ fixes some point of $V$. 
\item[(2)] 
The transformation $g$ is a reflection
if and only if %$s=\Lin(g)$ is a reflection of finite order with 
%reflecting hyperplane
%$H_s$
%perpendicular
%to $\Tran(s)$, i.e.,
$$
g(v)=s(v)+b\quad\text{ for all}\ v\in V
$$
for some central reflection $s\in\GL(V)$
of finite order and
$b\in H_{s}^{\perp}$.
Here,
$H_g=H_{s} + h$
for any $h\in H_g$.
\item[(3)]
The transformation $g$ is a reflection if
and only if\/ $g$ fixes a point of $V$ and $\Lin(g)$ is a reflection.
\end{itemize}
\end{lemma}
%%%%%%%%%%%%%%%%%%%%
\begin{proof}
  For (1), if $g$ has finite order, then it
  fixes the average of the elements in the orbit of the zero vector in $V$ under the cyclic group $\langle g
  \rangle$. 
  Conversely, suppose $g$ fixes a point $u$ in $V$ and $\Lin(g)$ has finite order. Then $u=g(u)=\Lin(g)u+\Tran(g)$ and $g(v)=\Lin(g)v+\Tran(g)=\Lin(g)(v-u)+u$ for all $v\in V$.  Then $g$ has finite order as $\left< g\right>$ is conjugate to the finite group
  $\langle \Lin(g) \rangle$. 
  
  For (2), 
  let $\Lin(g)=s$ and $\Tran(g)=b$.
  Suppose $g$ is a reflection.
 Then $H_g=H_0+c$ for some central hyperplane
  $H_0$ in $V$ and some $c$ in $V$, and for any $h_0\in H_0$,
  $$
  h_0+c=g(h_0+c)=s(h_0)+s(c)+b=s(h_0)+g(c)=s(h_0)+c.
  $$
  Hence, $s\neq 1$ is a reflection in $\GL(V)$ with $H_s=H_0$.
  Furthermore, 
  $b\in\text{Im}(1-s)=H_0^{\perp}$.
  Conversely, 
  assume $s\in\text{GL}(V)$
  is a reflection of finite order and $b\in H_s^{\perp}$. Then $s(b)=\xi b$ for some primitive $m$-th root-of-unity where $m$ is the order of $s$.
  Then $$g^m(v)=s^m(v)+\Sigma_{k=0}^{m-1}\ s^k(b)=v+\Sigma_{k=0}^{m-1}\ \xi^k b=v
  \quad\text{ for all } v \in V
  $$ 
  and $g$ also has finite order.  By part~(1), $g$ must fix a point $u$ in $V$. Then $g(h+u)=s(h+u)+b=s(h)+s(u)+b=h+u$ for any $h$ in $H_s$ and
  $g$ is a reflection about the hyperplane $H_s+u$.
  Part~(3) follows from part~(2).
\end{proof}

%%%%%%%%%%%%%%%%%%%%%%%%%%%%%%%%%%%%%%
\subsection*{Reflection groups}
 
An {\em affine reflection group}
(or just {\em reflection group})
is a subgroup of $A(V)$
generated by affine
reflections acting discretely on $V=\CC^n$. The {\em reflecting hyperplanes} for 
a group $W$ are the hyperplanes fixed
by reflections in $W$.
For any reflection group $W$, we set (following Popov~\cite{Popov})
$$
\Lin(W)=\{\Lin(g): g\in W\}\quad\text{ and }\quad
\Tran(W)=\{g\in W: g = \Tran(g)\}\, .
$$
Every affine reflection group $W$ is the product of irreducible
affine reflection groups
(or trivial groups)
and $W$ is irreducible
exactly when $\Lin(W)$ is
irreducible,
see~\cite{Popov}. If $G\leq \GL_n(\CC)$ is an
irreducible finite reflection group and $\Lambda$
is a $G$-invariant lattice in $V$, then $W=G\ltimes \Lambda$
is an affine reflection group with
$\Lin(W)=G$ and $\Tran(W)=\Lambda$.

%%%%%%%%%%%%%%%%%%%%%%%%%%%%%%%55
\subsection*{Crystallographic groups}
An affine reflection group $W$ 
acting on $V=\CC^n$ is {\em crystallographic} if its space of orbits $V/W$ is compact; otherwise it is {\em noncrystallographic}. 
Popov~\cite{Popov} showed that if
$W$ is an irreducible affine reflection group of infinite order,
then $\Tran(W)$ is a lattice of
rank $n=\dim V$ or $2n$,
and 
the crystallographic groups are those
whose lattices
have full rank $2n$.

%%%%%%%%%%%%%%%%%%%%%%%%%%%%%%%%%%%
\subsection*{Genuine groups}
A {\em Coxeter group} is a group of general linear transformations generated by reflections on $\RR^n$.
We assume all Coxeter groups
act discretely.
Every Coxeter group defines
a reflection group on $\CC^{n}$ by extension of scalars.
The infinite Coxeter groups acting
on $\RR^n$ are the {\em affine Weyl groups}; they define
affine reflection groups
acting on $\CC^n$ which are 
not crystallographic 
since the underlying lattices 
have rank $n$ instead of $2n$.
But the Weyl groups
acting on $\RR^n$
define groups acting on $\CC^n$
which stabilize various
lattices of full rank $2n$.  
Following Malle~\cite{Malle},
we call the
resulting affine complex
reflection groups $W$ {\em non-genuine}:
they are crystallographic but
have linear part $\Lin(W)$
a Coxeter group.
We say a crystallographic reflection group $W$
is {\em genuine} when $\Lin(W)$
is not merely
obtained as the complexification of a finite Coxeter group.
See Section~\ref{Coxetersection}.

%%%%%%%%%%%%%%%%%%%%%%%%%%%%%%%%%%%%%%%%%%%%%%%%%%%%%%%%5
%%%%%%%%%%%%%%%%%%%%%%%%%%%%%%%%%%%%%%%%%%%%%%%%%%%%%%%%%
\subsection*{Steinberg's fixed point theorem}
%\label{steinbergsthm}

We recall Steinberg's fixed point theorem:

\begin{theorem}
[Steinberg~\cite{Steinberg},
~\cite{Steinberg68}]
\label{thm:steinberg}
Let $W$ be a Coxeter group (of finite or infinite order)
or a finite complex reflection group
acting on\/ $V$.
Then a vector $v$ in $V$ is fixed by some nonidentity group element of\/ $W$
if and only if\/ $v$ lies on a reflecting hyperplane for $W$.
\end{theorem}

\begin{definition}
We say an affine reflection group $W$ 
has the {\em Steinberg property}
if 
the set of nonregular points 
is the union of
reflecting hyperplanes
for $W$, i.e.,
if the conclusion
of Theorem~\ref{thm:steinberg} holds.
\end{definition}

Recall that affine transformations
which are conjugate under
some $a$ in $\GL(V)$ 
have fixed point spaces
in the same $a$-orbit.
Thus to show that the
fixed point space $V^g$ of some element $g$
in a reflection group $W$ lies
on a reflecting hyperplane for
$W$, we may replace $g$ by
any conjugate of $g$ in $W$.
We may
also replace $g$ by
any power of $g$, as $V^g\subset V^{g^j}$ for all $j$.

%%%%%%%%%%%%%%%%%%%%%%%%%%%%%%%%%%%%%%%%%%%%%%%%%%%%%%%5
%%%%%%%%%%%%%%%%%%%%%%%%%%%%%%%%%%%%%%%%%%%%%%%%%%%%%5
%%%%%%%%%%%%%%%%%%%%%%%%%%%%%%%%%%%%%%5
\section{Classification}
\label{classification}
%%%%%%%%%%%%%%%%%%%%%%%%%%%%%%%%%%%%%%
\subsection{Rank 1 crystallographic groups}
\label{rank-one}
Every rank $n=1$ reflection group $W$ trivially satisfies
the Steinberg property
since affine hyperplanes
are just points in $V=\CC^1$
and the reflections of $W$ are exactly those group elements in $W$ fixing a point of $V$.
Our arguments later use orbits of various rank $1$ reflection groups acting on certain sets, so we give a few more
details on this case here.  Every rank $n=1$ crystallographic 
reflection group is $$W=G(r,1,1)\ltimes \Lambda $$
for some cyclic group
$G(r,1,1)=\langle \xi \rangle
\subset \text{GL}(V)$
acting on $V=\CC$
for $\xi=e^{\frac{2\pi i}{r}}$ with $r\geq 2$
and some lattice $\Lambda=\ZZ\oplus \ZZ\zeta$
stable under multiplication by $\xi$ with $\zeta\in\CC$. When $W$ is nongenuine, $r=2$
and $\Lambda$ is equivalent to
$\ZZ+\ZZ \alpha$ 
for some $\alpha\in \CC$ in the modular strip
(see Section~\ref{Coxetersection}).
When $W$ is genuine, only 2 possible lattices $\Lambda$ arise,
with $\zeta$ a third or fourth
root-of-unity in $\CC$,
and $r=3$, $4$ or $6$.
In this case, $\ZZ+\ZZ\zeta=\ZZ[\zeta]$
and $W\cong \ZZ/r\ZZ \ltimes \ZZ[\zeta]$.
In fact, every genuine
crystallographic
reflection group acting on $\CC$ is equivalent 
to a subgroup
of one of the 3 examples below
(see Popov~\cite{Popov}).

\begin{example}
The reflection group $W=G(4,1,1)\ltimes \BBZ[i]$ acting on $V=\CC$ is crystallographic
with $\frac{1}{2}\BBZ[i]$ the set of reflecting hyperplanes.
 Larger dots indicate
 points in the lattice $\ZZ[i]$.

 \begin{center}
  \includegraphics[scale=.025]{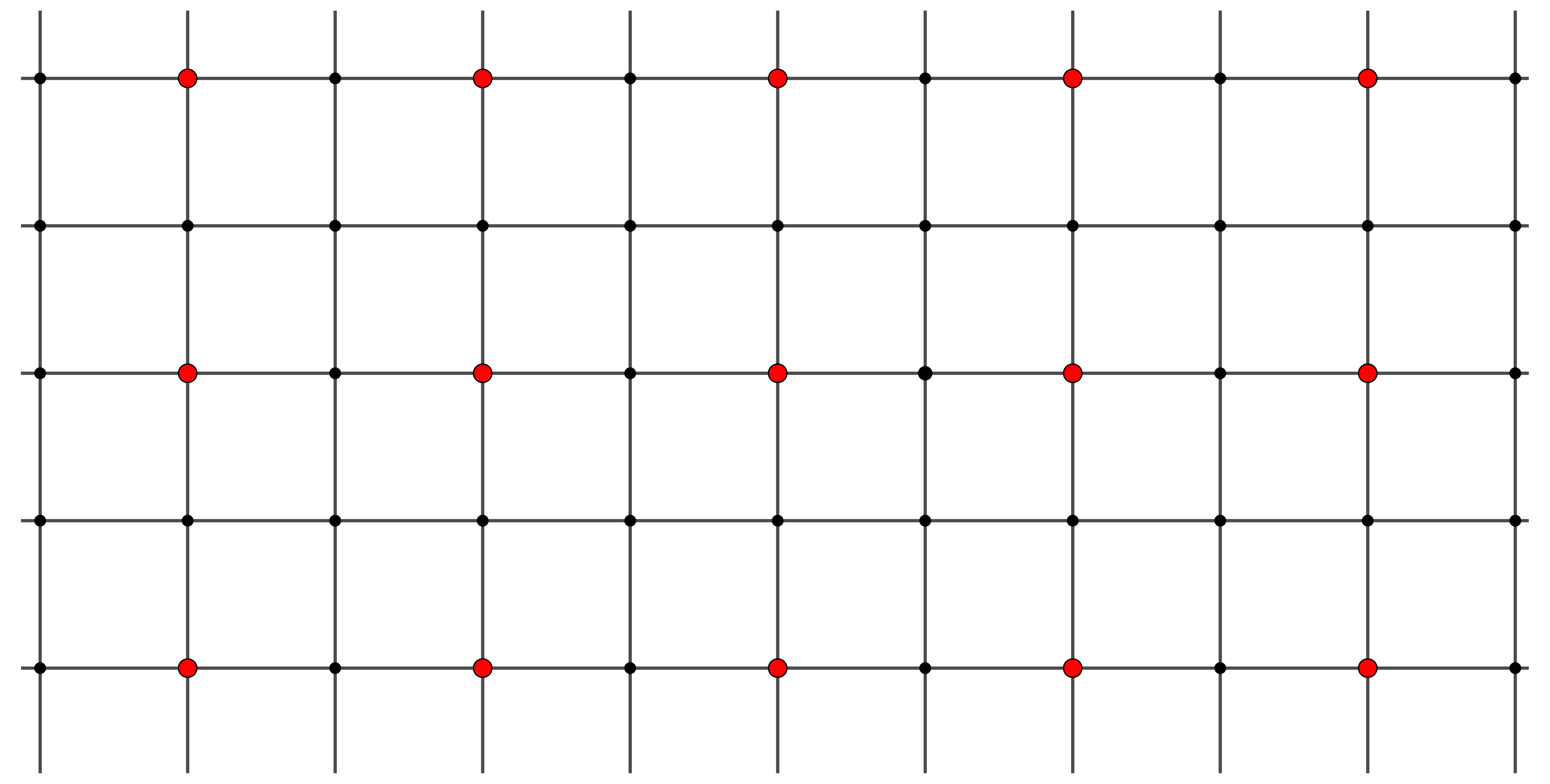}   
 \end{center}
 
\end{example}

\begin{example}
The reflection group $W=G(6,1,1)\ltimes\mathbb{Z}[\om]$ for $\om=e^{2\pi i/3}$ 
acting on $V=\mathbb{C}$ is crystallographic
with 
$\tfrac{1}{1-\om}(\ZZ+\ZZ\om)\cup\tfrac{1}{2}(\ZZ+\ZZ\om)
$
the set of reflecting hyperplanes.
Again, 
larger dots
indicate points in the lattice
$\ZZ[\om]$.
\begin{center}
\includegraphics[scale=.1]{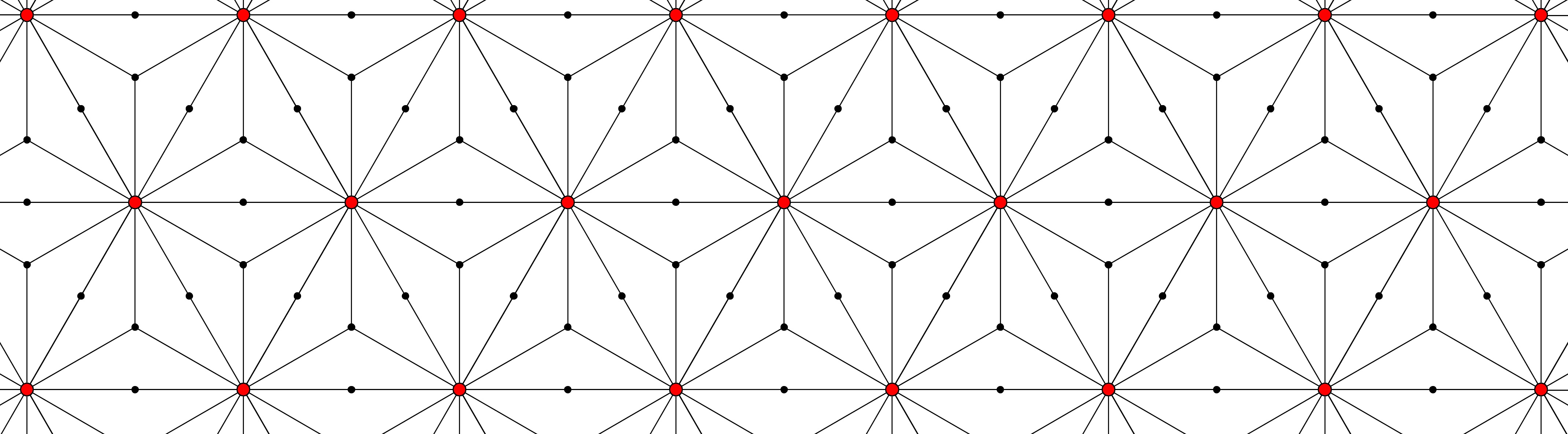}
\end{center}
\end{example}

\begin{example}
The group
$W=G(3,1,1)
\ltimes\mathbb{Z}[e^{2\pi i/3}]$ is a crystallographic reflecting
group acting on $V=\CC$.
\end{example}
%%%%%%%%%%%%%%%%%%%%%%%%%%%%%%%%%%%%%%%%%%%%%%%%%%%%%%%%%5
\vspace{3ex}

\subsection*{The 3-parameter family of finite complex reflection groups}

Shephard and Todd~\cite{ST}
classified the irreducible
finite complex reflection groups.
They give a 3-parameter
family $G(r,p,n)$ and
34 exceptional groups denoted by $G_i$ for $4\leq i\leq 37$. 
The group $G(r,1,n)$
consists of $n \times n$
monomial matrices (i.e., matrices with a single 
nonzero entry in each row and column) whose nonzero
entries are complex $r$-th roots-of-unity, for $r$ a positive integer.
Note that as an abstract group, 
$$G(r,1,n)\cong \mathfrak{S}_n\ltimes(\mathbb{Z}/r\mathbb{Z})^n,$$
where $\mathfrak{S}_n$
is the symmetric group.
Each group $G(r,1,n)$ is generated
by reflections on $V=\CC^n$ 
of order $2$ and order $r$
(see~\cite{OrlikTerao}).
In fact, $G(r,1,n)$ is the symmetry group of
the {\em cross-polytope} in $\CC^n$,
a regular complex polytope
studied
by Shephard~\cite{Shephard} and Coxeter~\cite{Coxeter}.
Note that $G(r,1,n)$ acts
by isometries with respect
to the standard inner product.

For any integer $p \geq 1$ dividing $r$,
the group $G(r,p,n)$ is the subgroup of $G(r,1,n)$
consisting of those matrices
whose product of nonzero entries is $1$ when raised to the power $r/p$.
The groups $G(r,p,n)$ are also generated by reflections and
include the infinite families of Coxeter groups 
acting on $\RR^n$:
\begin{itemize}
\item 
$G(2,1,n)$ is the Weyl group $\text{W}(B_n)$,
\item $G(2,2,n)$
is the Weyl group
$\text{W}(D_n)$,
\item  $G(r,r,2)$ is the dihedral group $\text{W}(I_r)$ 
of order $2r$ after change-of-basis,
\item  $G(1,1,n)$ is the symmetric group $\mathfrak{S}_n$ acting by permutation matrices
with irreducible reflection
representation
$W(A_{n-1})$.
\end{itemize}
Note that
$G(3,3,2)$ is equivalent to the  complexification
of the Weyl group $W(A_2)$.
Also, $G(4,4,2)$ and $G(2,1,2)$ are equivalent and
$G(2,2,2)=G(2,1,1)\times G(2,1,1)$.

\begin{example}
The group $G=G(4,1,2)$ is generated by matrices 
$
\left(\begin{smallmatrix} 
i&0\\0 &1 
\end{smallmatrix}\right),
\left(\begin{smallmatrix} 0&1\\1 &0 
\end{smallmatrix}\right)
$.
The reflections in $G$ 
are the diagonal matrices with only one diagonal entry not equal to 1 
and the antidiagonal matrices whose nonzero entries are inverse.
The reflecting hyperplanes for $G$ are $H_j=\ker(x_j)$ 
for $j=1,2$ and $H_{1,2}(\zeta)=\ker(x_1-\zeta x_2)$ 
for $\zeta$ a $4$-th root-of-unity.
\end{example}

%%%%%%%%%%%%%%%%%%%%%%%%%%%%%%%%%%%%%%%%%%%%%%%%%%%%%%%%
\subsection*{Popov's Classification of crystallographic groups}

Popov~\cite{Popov} classified the  crystallographic reflection groups
$W$ using Shephard
and Todd's~\cite{ST} notation for their linear parts,
showing that 
if $W$ is irreducible, then $\Lin(W)$ fixes $\Tran(W)$ 
set-wise
and $\Lin(W)$ is one of  
$$ W(A_{n-1}),\ G(2,p,n),\ G(3,p,n),\ G(4,p,n),\
\text{or } G(6,p,n),$$ 
or one of the $16$ exceptional groups
$$
G_4,\ G_5,\ G_8,\ G_{12},\ G_{24},\ G_{25},\ G_{26},\ G_{28},\ G_{29},\ G_{31},\ G_{32},\ G_{33},\ G_{34},\ G_{35},\ G_{36},\ G_{37}.
$$
Each nongenuine crystallographic complex reflection group
stabilizes a moduli space of full rank lattices in $\CC^n$,
see Section~\ref{Coxetersection}.
If $W$ is an irreducible crystallographic complex reflection group
and $W\not\cong \Lin(W)\ltimes \Tran(W)$, 
then 
$\Lin(W)$ is $G(4,2,n)$, $G(6,2,2)$, $G_{12}$, or $G_{31}$.
(Goryunov~\cite{Goryunov} gives the group $[G(6,2,2)]^*$ 
left out of the classification.)
We do not consider these exceptional
cases here.

Popov~\cite{Popov} determined that each irreducible genuine finite complex reflection group $W$
stabilizes at most three 
full rank lattices, up to equivalence.
His notation 
$$
W=[G_i]_k
$$ 
indicates that $W$
has linear part $\Lin(W)=G_i$
in the notation of  Shephard and Todd with the index $k=1$, $2$, or $3$ 
indicating one of possibly
three different lattices
$\Tran(W)=\Lambda$
stabilized by $G_i$.
When there is only one such
lattice, we write $[G_i]_1$,
although Popov merely writes $[G_i]$.

\subsection*{4-parameter family
of crystallographic groups}
The $3$-parameter family 
of finite groups $G(r,p,n)$ 
gives rise to a $4$-parameter family of crystallographic reflection 
groups 
$[G(r,p,n)]_k=
G(r,p,n)\ltimes \Lambda$ where $\Lambda$
is a $G(r,p,n)$-invariant
lattice of rank $2n$ in $\CC^n$.
The affine reflections in
$G(r,p,n)\ltimes \Lambda$ have the form
$
s(v)=\sigma(v) + b$
for $v\in V$
where $\sigma$ is a reflection in
$G(r,p,n)$ fixing some central hyperplane $H_\sigma$
of the form $\ker(x_j-\xi^p x_k)$ or $\ker(x_j)$ for $1\leq j<k\leq n$
and where $b\in \Lambda
\cap (H_\sigma)^\perp$
(see Lemma~\ref{geometrylemma}).

Table~\ref{StateTable} 
gives the
genuine crystallographic 
groups in the infinite family
$[G(r,p,n)]_k=G(r,p,n)\ltimes \Lambda$
and specifies 
those for which
Steinberg's theorem holds.
Here,
$k=1,2,3$ indicates choice of lattice $\Lambda$.
We omit non-genuine groups
$W$, as they appear in
Section~\ref{Coxetersection},
and thus we assume
$r>2$ and omit
the groups $[G(r,r,2)]$.
Again, $\xi$ is a primitive $r$-th root-of-unity
in $\CC$;
note that
$\ZZ[\xi]=\ZZ+\ZZ\xi$ for $r=3,4,6$.
\newcommand{\strutB}{\rule[-2ex]{0ex}{5ex}}
\begin{small}
\begin{table}[h!tbp] % h! = put it here
\caption{Genuine crystallographic
    complex reflection groups\\
    ${}_{}\hspace{4ex}
    W=[G(r,p,n)]_k=G(r,p,n)\ltimes\Lambda$\ \ acting on $\CC^n$}
\label{StateTable}
  \centering
\begin{tabular}{p{.15\textwidth}p{.08\textwidth}p{.42\textwidth}p{.15\textwidth}
}
  \toprule
{\bf Group $W$}
& {\bf dim} 
& {\bf $G(r,p,n)$-invariant lattice $\Lambda$}
& {\bf Steinberg's thm}  
\\
\midrule
\strutB
$[G(r,1,1)]_1$
& $n=1$ 
& $\ZZ[\xi]e_1$
& \cmark 
\\ %\hline 
\midrule
\strutB
$[G(3,1,n)]_1$
& $n\geq 2$
& $\ZZ[\xi]e_1+\Sigma_{j=2}^{n}\ZZ[\xi](e_{j-1}-e_j)$
& \cmark 
\\ %\hline 
\strutB
$[G(3,1,n)]_2$
& $n\geq 2$
& $\ZZ[\xi]e_1+\Sigma_{j=2}^{n}\ZZ[\xi]
\left(\tfrac{1}{1-\xi}\right)(e_{j-1}-e_j)$
& \xmark 
\\ 
\strutB
$[G(3,3,n)]_1$ 
& $n\geq 3$
&  $\ZZ[\xi](\xi e_1-e_2)+\Sigma_{k=j}^{n}\ZZ[\xi](e_{j-1}-e_j)$
& \xmark 
\\ %\hline
\midrule
\strutB
$[G(4,1,n)]_1$
& $n\geq 2$
& $\ZZ[\xi]e_1+\Sigma_{k=j}^{n}\ZZ[\xi](e_{j-1}-e_j)$
& \cmark 
\\ %\hline 
\strutB
$[G(4,1,n)]_2$
& $n\geq 2$
& $\ZZ[\xi]e_1+\Sigma_{j=2}^{n}\ZZ[\xi]\left(\tfrac{1}{1-\xi}\right)(e_{j-1}-e_j)$
& \cmark 
\\ %\hline 

\strutB
$[G(4,2,n)]_1$
& $n\geq 2$
& $\ZZ[\xi](\xi e_1-e_2)+\Sigma_{j=2}^{n}\ZZ[\xi](e_{j-1}-e_j)$
& \cmark 
\\ %\hline 
\strutB
$[G(4,2,n)]_2$
& $n\geq 2$
& $\ZZ[\xi](\xi e_1-e_2)+\Sigma_{j=2}^{n}\ZZ[\xi](e_{j-1}-e_j)+\ZZ[\xi]e_n$
& \cmark 
\\ %\hline 
\strutB
$[G(4,2,2)]_3$
& $n= 2$
& $\ZZ[\xi](\xi e_1-e_2)+\ZZ[\xi](1+\xi)(e_1-e_2)$
& \cmark 
\\ %\hline 
\strutB
$[G(4,4,n)]_1$
& $n\geq 3$
&  $\ZZ[\xi](\xi e_1-e_2)+\Sigma_{j=2}^{n}\ZZ[\xi](e_{j-1}-e_j)$
& \xmark 
\\ %\hline
\midrule
\strutB
$[G(6,1,n)]_1$
& $n\geq 2$
& $\ZZ[\xi]e_1+\Sigma_{j=2}^{n}\ZZ[\xi](e_{j-1}-e_j)$
& \cmark 
\\ %\hline
\strutB
$[G(6,2,n)]_1$
& $n\geq 2$
& $\ZZ[\xi](\xi e_1-e_2)+\Sigma_{j=2}^{n}\ZZ[\xi](e_{j-1}-e_j)$
& \cmark 
\\ %\hline 
\strutB
$[G(6,3,n)]_1$
& $n\geq 2$
&  $\ZZ[\xi](\xi e_1-e_2)+\Sigma_{j=2}^{n}\ZZ[\xi](e_{j-1}-e_j)$
& \cmark 
\\ %\hline 
\strutB
$[G(6,2,2)]_2$
& $n= 2$
&$\ZZ[\xi](\xi e_1-e_2)+\ZZ[\xi](1+\xi)(e_1-e_2)$ 
& \cmark 
\\ %\hline 
\strutB
$[G(6,3,2)]_2$
& $n= 2$
& 
$\ZZ[2\xi](\xi e_1-e_2)+ \ZZ[2\xi](1-\xi)(e_1-e_2)$
& \xmark 
\\ %\hline 

\strutB
$[G(6,6,n)]_1$
& $n\geq 3$
&  $\ZZ[\xi](\xi e_1-e_2)+\Sigma_{j=2}^{n}\ZZ[\xi](e_{j-1}-e_j)$
& \xmark 
\\ \bottomrule
\end{tabular}
\end{table} 
\end{small}
%%%%%%%%%%%%%%%%%%%%%%%%%%%%%%%%%%%%%%%%%%%%%%%%%%%%%%%%%%

\FloatBarrier
%\clearpage %print any pages that have not yet appeared.
%%%%%%%%%%%%%%%%%%%%%%%%%%%%%%%%%%%%%%%%%%%%%%%%%%%%%%%%%

\section{The symmetric group}
\label{symmetricgroupsection}
The group $G(1,1,n)$
is the symmetric group $\mathfrak{S}_n$ acting
on $V=\CC^n$ in its natural reflection representation by permutation of
basis vectors $e_1,\ldots, e_n$.
It is reducible and not the linear part of any
crystallographic affine
reflection group.
We identify the irreducible Weyl group $W(A_{n-1})\cong \mathfrak{S}_n$ with the restriction of 
$G(1,1,n)$ to 
the subspace  $V'=\CC\text{-span}\{e_2-e_1, \ldots,
e_n-e_{n-1}\}\cong \CC^{n-1}$ of $V$.
Every irreducible
crystallographic complex reflection group 
with linear part 
$W(A_{n-1})$
lies in the $1$-parameter
family
%for parameter
%$\alplha\in \CC$,
$$
[W(A_{n-1})]^{\alpha}_1
=W(A_{n-1})\ltimes \Lambda_\alpha
\cong \mathfrak{S}_n
\ltimes \Lambda_\alpha
\quad\text{ for }\ 
%$$
%f5or $\Lambda_\alpha$
%the lattice of rank
%$2n$ in $\CC^n$
%$$
\Lambda_\alpha
\ =\ \Sigma_{j=2}^{n}(\ZZ+\ZZ\alpha)(e_{j-1}-e_j),
\ \ \text{ with parameter }\alpha\in \CC.
$$
In fact, every $W[(A_{n-1})]_1^\alpha$
is equivalent to a group
with $\alpha$ in the modular strip
(see Section~\ref{Coxetersection}).

We use the next proposition to streamline
arguments later for genuine
and nongenuine groups.
\begin{lemma}
\label{cyclecase}
Let $W=G(r,p,n)\ltimes \Lambda$
for some $G(r,p,n)$-invariant
lattice $\Lambda$ of full rank
in $\CC^n$, for $r,p\geq 1$,
$n\geq 2$.
Suppose that $g$ in $W$ 
has linear part a nontrivial cycle in
the symmetric group
$G(1,1,n)$.
Then any vector fixed by $g$ lies on a reflecting hyperplane for $W$.
\end{lemma}
\begin{proof}
The claim follows for
$n=1$, see Subsection~\ref{rank-one},
so we assume $n>1$.
After conjugation,
we may assume
$\Lin(g)=(m\  \ m+1\ 
\ m+2\ \cdots\ \ell)\neq 1$ in $\mathfrak{S}_n$ for some $1\leq m\leq n$, using
$\mathfrak{S}_n$-notation
for elements of $G(1,1,n)$.
Let $s$ be the affine reflection about the hyperplane
$H_s=\ker(x_m-x_{m+1}+\beta)$ in $V$ defined by
$$
s(v) = (m\ \ \ m+1)v + \beta (e_{m+1}- e_{m})
\quad\text{for}\ \ v \in V\, ,
$$
with $\beta=x_{m+1}(\Tran(g))\in \CC$. 
Then $\Tran(g)\in\Lambda$ implies $\beta (e_{m+1}- e_{m})\in \Lambda$ as well, upon inspection
of the possible lattices,
see Tables~\ref{StateTable} and~\ref{CoxeterTable},
and $s$ lies in $W$.
If $g$ fixes $u\in V$, then
$x_{m+1}(u)=x_m(u) + \beta$ and 
$u$ lies on 
the reflecting hyperplane 
$H_s$ of $W$.
\end{proof}

In the next proposition,
we see that Steinberg's theorem
holds for
all reflection groups
of the form
$W=W(A_{n-1})
\ltimes \Lambda$
for any
 $W(A_{n-1})$-invariant lattice 
 $\Lambda$ of full rank in $\CC^{n-1}$.
 
\begin{prop}
\label{symmetric}
Let $W$ be a crystallographic complex reflection group
whose linear part is the symmetric group $\mathfrak{S}_n$ in its
irreducible reflection representation
as the Weyl group 
$W(A_{n-1})$.
Then $W$ satisfies the Steinberg property.
\end{prop}
\begin{proof}
Consider
$W'=W(A_{n-1})\ltimes \Lambda_\alpha$ acting on $V'$
for some $\alpha\in \CC$
and let
$W=G(1,1,n)\ltimes \Lambda$ 
for $\Lambda$
the rank $2(n-1)$
lattice $\Lambda_\alpha$
regarded as a lattice
in $V =\CC^n$.
We follow the proof of the last lemma to see that  any vector
in $V$ fixed
by a group element of $W$ lies on a reflecting hyperplane for $W$.
Since $W$ is the direct sum of $W'$
and the trivial rank $1$
representation of $\mathfrak{S}_n$,
the claim follows for 
$W'$ as well.
\end{proof}
%%%%%%%%%%%%%%%%%%%%%%%%%%%%%%%%%%%%%%%%%%%%%%%%%%%%%%%%%%%%
\section{Which groups have the Steinberg property?}
\label{groupspassing}
In this section, we determine those genuine
crystallographic
complex reflection groups  $G(r,1,n) \ltimes \Lambda$ which satisfy Steinberg's theorem.
We begin with a small lemma
that does not require
that $W$ be genuine. 
The lemma allows us to later reduce to arguments
on orbits under the action
of $1$-dimensional reflection groups.
In part~(3) below, we consider
orbits under the action of 
$G(r,1,1)\ltimes \Lambda'$ on $\CC$, with $G(r,1,1)=\langle \xi \rangle$ acting by multiplication
and
$\Lambda'\subset \CC$
acting by translation (identifying $\CC e_1$ with $\CC$). 
%%%%%%%%%%%%%%%%%%%%%%%%%%%%%%%%%%%%%
%%%%%%%%%%%%%%%%%%%%%%%%%%%%
\begin{lemma}
\label{TechnicalLemma}
Let $W=G(r,p,n)\ltimes \Lambda$
for some
$G(r,p,n)$-invariant lattice $\Lambda$ in $\CC^n$. Consider any $g$ in $W$ and write
$\lambda_1,\ldots,
\lambda_n$ for the diagonal entries of\/ $\Lin(g)$.
Suppose one of the following holds:
\begin{itemize}
\item [(1)]
$\lambda_j
\not\in\{0, 1\}$ 
for some index $j$ and
$g^p\neq 1$ with $r\neq p$ and $ x_j(\Tran(g^p)) e_1\in \Lambda$;
\item [(2)]
$\lambda_j
\not\in\{0, 1\}$ 
for some index $j$ and $g^p=1$ with $r\neq p$ and
$x_j(\Tran(g))\, e_1\in \Lambda\ (1-\lambda_j)
\rule[-2ex]{0ex}{4ex}$;
\item[(3)]
$\lambda_j,
\lambda_\ell\not\in\{0, 1\}$  
for some $j\neq \ell$,
the lattice $\Lambda$
is invariant under $G(r,1,n)$, and
$$
\tfrac{x_j(\Tran(g))}{1-\lambda_j}\ 
\quad\text{ and }\quad \tfrac{x_\ell(\Tran(g))}
{1-\lambda_\ell}\ 
$$
lie
in the same orbit
under the action of\/
$G(r,1,1)\ltimes \Lambda'
$ on\/ \ $\CC$\/
for\/
$\Lambda'= x_1
\big(\CC(e_1-e_2)\cap\Lambda\big)$.
\end{itemize}
Then any vector fixed
by $g$ lies on a reflecting hyperplane of\/ $W$.
\end{lemma}
%%%%%%%%%%%%%%%%%%%%%%%%%%%%%%%%%%%%%%%
\begin{proof}
We use Lemma~\ref{geometrylemma} throughout.
Suppose~(1) holds and
set $\beta=x_j(\Tran(g^p))$.
Any vector $u$ in $V$ fixed 
by $g$ is fixed by
$g^p$ and thus
satisfies $x_j(u)=\lambda^p_j x_j(u)+\beta$,
since
$\Lin(g^p)=(\Lin(g))^p$.
Hence $u$  
is fixed by the 
affine transformation $s$
defined by
$$
s(v)=\text{diag}(1,\ldots,1,
\lambda^p_j, 1,\ldots, 1)v+
\beta
e_j
\quad\text{ for } v \in \CC^n,
$$
with $\lambda^p_j$ as $j$-th entry.
Then $\Lin(s)\neq 1$ lies in $G(r,p,n)$
and $\Tran(s)=\beta e_j$
lies in $\Lambda$
since $\beta e_1$ lies
in $\Lambda$
and $\Lambda$ is invariant
under $\Lin(W)$. 
As the central reflecting hyperplane for $\Lin(s)$ is perpendicular to $\Tran(s)$, the transformation $s$ is in fact an affine reflection in $W$.

%%%%%%%%%%%%%%%%%%%%%%%%%%%%%%%%%%%%%%

Now suppose (2) holds.
Any vector fixed by $g$ is
fixed by the affine reflection $s$ defined by
$$
s(v)=\text{diag}(1,\ldots,1,
\xi^p, 1,\ldots, 1)v+\beta e_j
\quad\text{ for } v \in \CC^n,
$$ 
with $\xi^p$ as $j$-th entry
and
$\beta=
\frac{1-\xi^p}
{1-\lambda_j}x_j(\Tran(g))
$.
Then $\Lin(s)$ lies in $ G(r,p,n)$. 
We set $\beta'=(1-\xi^p)^{-1}\beta$.
Then as
$\beta'
e_1$ lies in $\Lambda$, 
both $\beta'e_j$
and 
$\xi^p \beta'e_j$ lie in $\Lambda$ as well since  
the lattice $\Lambda$ is $\Lin(W)$-invariant.
Hence their difference,
$\Tran(s)=\beta e_j=
(1-\xi^p)\beta' e_j$, also lies in $\Lambda$
and $s$ is an affine reflection in $W$.
%%%%%%%%%%%%%%%%%%%%%%%%5

Lastly, say~(3) holds.
Since
the two given quotients
are in the same orbit,
their weighted difference
$$
\alpha\ =\
\tfrac{x_j(\Tran(g))}{1-\lambda_j}
\ -\ 
\xi^{m}\ \tfrac{x_\ell(\Tran(g))}{1-\lambda_\ell}
$$
lies in $\Lambda'$
for some $m\geq 0$,
i.e.,
$\alpha (e_1-e_2)\in\Lambda$. 
Since $\Lambda$ is $G(r,1,n)$-invariant,
$\alpha(
e_j-\xi^{-m}e_\ell)$
lies in $\Lambda$ as well.
Let $s$ be the affine transformation defined by
$$
s(v) = \sigma(v)
%t_j^m t_\ell^{-m}(j\ \ k)
+ \alpha(
e_j-\xi^{-m}e_\ell)
\quad\text{ for  }
v\in V,
$$
where $\sigma\in G(r,p,n)$ is the weighted transposition 
$e_j\mapsto \xi^{-m} e_\ell$
and $e_\ell \mapsto \xi^{m}e_j$. Then $s$ is an affine reflection in $W$
since
$\Tran(s)$
is perpendicular to the
central reflecting hyperplane $H_\sigma$ of
$\sigma$.
We argue that any vector fixed 
by $g$ is fixed by $s$. 
Indeed, if $u\in V$ is fixed by $g$,
then 
$$
x_j(u)=\tfrac{x_j(\Tran(g))}
{1-\lambda_j}
\quad\text{and}\quad
x_\ell(u)=\tfrac{x_\ell(\Tran(g))}
{1-\lambda_\ell}\ .
$$
Since $\alpha=x_j(u)-\xi^m x_\ell(u)$, the vector
$u$ is fixed by $s$.
%%%%%%%%%%%

In all three cases, we see that any vector fixed by $g$ lies on an affine reflecting hyperplane for $W$.
\end{proof}

The next example shows
how we use use
orbits in $\CC$
to determine fixed point
spaces.  We also use this example
in the next proof.

%%%%%%%%%%%%%%%%%%%%%%%%%
\begin{example}
\label{[G(4,1,n)]_2}
Consider $W=[G(4,1,n)]_2$
with $n>1$ and some
$g$ in $W$ with nontrivial
fixed point space $V^g$
in $V=\CC^n$.  Suppose 
$\Lin(g)$ has two
diagonal entries
$\lambda_j$
and $\lambda_\ell$
that are 
$-1$. We claim that $V^g$
lies on a reflection hyperplane
for $W$.
Since $\Lin(g)$ is monomial,
$$
g(e_j)=- e_j+\beta_j
\quad\text{ and }\quad
g(e_\ell)=- e_\ell +\beta_\ell
$$
where $\beta_j=x_j(\Tran(g))$ and
$\beta_\ell=x_\ell(\Tran(g))$.
Here,
$\beta_j, \beta_\ell$
lie in the set
$\ZZ[i]/(1-i)=
\{(a+bi)/2: a,b
\text{ both even or
both odd}\}$.
If $\beta_j$ or $\beta_\ell$ lie in $\ZZ[i]$, then Lemma~\ref{TechnicalLemma}(1) 
implies the claim.
Hence we assume
$\beta_j$ and $\beta_\ell$ are both not in $\ZZ[i]$.
Then 
each can be written 
in the form 
$(a+bi)/2$ with $a,b$ both odd and thus identified
with the vertex $(a,b)$
of a square
in a tessellation
of the plane $\RR^2$
by squares of side length one, with one square centered at zero.  
Any vertex in this tessellation
can be obtained from any other
by an even number of horizontal and
 vertical translations (by one unit) together with some
rotations
of $90^{\circ}$.
Thus $\beta_j/2$ and $\beta_\ell/2$ 
lie in the same orbit
under the action of
$G(4,1,1)\ltimes
\Lambda'$
for\/ $\Lambda'
=
\ZZ[i]/(1-i)$.
Lemma~\ref{TechnicalLemma}(3) then implies the claim.
\end{example}

The next proposition
begins the analysis of
genuine crystallographic groups.

%%%%%%%%%%%%%%%%%%%%%%%%%%
\begin{prop}
\label{[G(r,1,n)]_1}
The groups 
$[G(r,1,n)]_1$ and $[G(4,1,n)]_2$\/
for $r\geq 3$, $n\geq 1$ have the Steinberg property.
  \end{prop}
  %%%%%%%%%%%%%%%%%%%%%%%%%
\begin{proof} 
Let $W$ be one of the given groups.
The claim follows for
$n=1$, see Subsection~\ref{rank-one},
so we assume $n>1$.
Fix $g\neq 1$ in $W$
with nontrivial fixed point space $V^g$.  We assume $g$ itself is
not a reflection,
nor any power of $g$,
so $\Lin(g)$
and its powers are
also non-reflections by Lemma~\ref{geometrylemma}.  Note that $\Lin(g)\neq 1$.

{\bf Diagonal action.}
First suppose that $\Lin(g)$ is diagonal.
Then at least two
diagonal entries are
non-trivial.
In case $W=[G(r,1,n)]_1$, $\Tran(g)$
lies in $(\ZZ[\xi])^n$
(by inspection of $\Lambda
=\Tran(W)$)
and
Lemma~\ref{TechnicalLemma}(1)
implies that $u$
lies
on a reflecting hyperplane
for $W$.
In case $W=[G(4,1,n)]_2$,
we may replace $g$ by $g^2$
if warranted (as $V^g\subset V^{g^2})$
and assume at
least two diagonal entries of $\Lin(g)$ are $-1$.
Then by Example~\ref{[G(4,1,n)]_2}, any vector fixed
by $g$ lies on a reflecting hyperplane for $W$.

%%%%%%%%%%%%%%%%%%%%%%%%%%%%%%%5
{\bf Cycle decomposition.}
Now assume
$\Lin(g)$ is not diagonal.
Every element of $G(r,1,n)$ is conjugate by an element of the
symmetric group $\mathfrak{S}_n$
to a disjoint product of nontrivial $\xi$-weighted cycles
(see~\cite{RamShepler})
$$
c=t_m^{a_m} t_\ell^{a_\ell} \
(m\ \ \ m+1\ \ m+2\ \ \cdots\ \ \ell)\, ,
\quad
$$ 
where each $t_m=\text{diag}(1, \ldots, 1, \xi, 1\, \ldots, 1)$, 
a diagonal matrix with $\xi$ as the $m$-th entry, 
and $a_m, a_\ell\in \ZZ_{\geq 0}$. 
Here, we identify
the cycle $(m\ \ m+1\ \cdots\ \ell)$ in $\mathfrak{S}_n$
with the corresponding permutation matrix in $G(1,1,n)$.  
We write $\Lin(g)$
as a product of such weighted cycles and fix attention
on one such
weighted cycle $c$ 
with $(m \ \ m+1\ \cdots\ \ell)$ nontrivial.

 {\bf Symmetric group action.}
Suppose that $\det(t_{m}^{a_m}\ t_\ell^{a_\ell})=1$. Then $c$ is conjugate
by $t_m^{-a_\ell}$
to the cycle $(m\ \ m+1\ \cdots \ \ell)$ 
as $1=\xi^{a_m+a_\ell}$
and we may assume $c$ is this cycle.
Then since $V^g \subset V^c$,
Lemma~\ref{cyclecase} implies
$V^g$ lies on a reflecting
hyperplane for $W$.

{\bf Diagonal power.}
Now suppose $\det(t_{m}^{a_m}\ t_{\ell}^{a_\ell})=\lambda\neq 1$.
Then $c^{m-\ell+1}$ in $G(r,1,n)$ is diagonal, as $m-\ell+1$ is the length of the cycle, and
$\Lin(g^{m-\ell+1})
=(\Lin(g))^{m-\ell+1}$ is block
diagonal with one
block itself a scalar matrix $\lambda I$ of size at least $2\times 2$.  The arguments for
the case when $\Lin(g)$ is diagonal above
then show that any
vector fixed by
$g^{m-\ell+1}$, and thus any fixed by $g$, lies
on a reflecting
hyperplane for $W$.
\end{proof}

\vspace{1ex}

\begin{remark}{
The idea for showing that
$[G(4,1,n)]_2$ has the Steinberg property in
the above proof 
of Theorem~\ref{[G(r,1,n)]_1}
does not apply 
to $[G(3,1,n)]_2$ as
the corresponding
tessellation 
of the plane $\RR^2$ 
from Example~\ref{[G(4,1,n)]_2}
has two orbits 
of vertices
under the $1$ dimensional group in Lemma~\ref{TechnicalLemma}(3). In fact, the group $[G(3,1,n)]_2$ does {\em not}
have the Steinberg property,
and Proposition~\ref{thosewithout}
below corrects a claim in~\cite{Puente}.}
\end{remark}

We find that even 
crystallographic groups 
of rank $2$ may require
some special analysis.  In the proof of the next proposition, we 
examine $1$-dimensional sublattice
structure.

\begin{prop}
\label{[G(6,2,2)]_2}
The group
$[G(6,2,2)]_2$ has the Steinberg property.
\end{prop}
\begin{proof}
Suppose $u$
in $V$
is fixed by some
nonreflection 
$g\neq 1$ in $W=[G(6,2,2)]_2$.
Let $\om=\xi^2=e^{2\pi i/3}$.
One may check that for any
$\gamma$ in $\ZZ[\om]=\ZZ[\xi]$, $W$
includes reflections about affine hyperplanes
\begin{equation}\label{somehyperplanes}
x_1=\gamma,\quad
x_2=\gamma,\quad
x_1-\xi^j x_j = \gamma\ \
\text{ for $j$ odd},\quad
x_1-\xi^j x_j= (1-\om)\gamma\ \ \text{ for  $j$ even}.
\end{equation}
We argue that $u$ lies on one of these hyperplanes
by applying the eight linear forms
$x_1, x_2, x_1-\xi^j x_2$
to $u$ and showing one of the resulting values lies in
$\ZZ[\om]$ 
or in $(1-\om)\ZZ[\om]$
(for $j$ even).

By Lemma~\ref{geometrylemma}(c),
$\Lin(g)$ 
is not
a reflection
and we may assume that
$\Lin(g)$ is $\diag{\omega,\omega^2}$, 
$\diag{-1,-1}$, 
or $\diag{\omega^2,
\omega^2}$
by replacing $g$
by a power of $g$ if necessary. Fix $\alpha,\beta\in\ZZ[\omega]$ with
$$\Tran(g)=\alpha\begin{pmatrix} -\omega^2\\ -1 \end{pmatrix}+\beta(1-\omega^2)\begin{pmatrix} 1\\ -1 \end{pmatrix}. $$

First, suppose $\Lin(g)=\diag{\omega,\omega^2}$. Then
$x_1(u)+\omega x_2(u)=\beta\in \ZZ[\om]$ and
$u$ lies on a reflecting hyperplane
of~(\ref{somehyperplanes}).

Second, suppose
$\Lin(g)=\diag{-1,-1}$.
If $x_2(u)\in\ZZ[\omega]$,
then $u$ lies on a reflecting hyperplane
of~(\ref{somehyperplanes}),
so we assume $x_2(u)\notin\ZZ[\om]$.
As $u$ is fixed by $g$,
we claim
applying the five linear forms
$x_1, x_2, x_1+x_2, x_1+\om x_2, x_1+\om^2 x_2$ to $u$
gives numbers in $(1/2)\ZZ[\om]$
which lie in distinct cosets of the subgroup $\ZZ[\om]$ 
of the additive group $(1/2)\ZZ[\om]$.
Indeed, overlapping cosets would imply that $x_2(u)=(1/3)(1-\om^2)(1-\om)x_2(u)$
lay in $(1/3)\ZZ[\om] \cap (1/2)\ZZ[\om]=\ZZ[\om]$.
As the subgroup $\ZZ[\om]$
has index $4$ in the larger group,
one of these cosets is trivial.
Thus 
$x_1(u)$, $x_2(u)$, $x_1+x_2(u)$, $x_1+\omega x_2(u)$,  or $x_1+ \omega^2x_2(u)$
lies in $\ZZ[\om]$
and
$u$ lies on some reflecting hyperplane
of~(\ref{somehyperplanes}).

Third, suppose $\Lin(g)=\diag{\omega^2,\omega^2}$. 
%Then
%$x_1(u)=\alpha(1-\omega^2)/3+\beta$ and $x_2(u)=\alpha(1-\omega)/3-\beta$ as $g$ fixes $u$.
We apply the three linear forms
$x_1-x_2, x_1-\om x_2, x_1-\om^2 x_2$ to $u$ and obtain
complex numbers
$\alpha+2\beta$, $-\om^2(\alpha+\beta)$,
$-\om\beta$.
We argue that one of
these numbers lies in $(1-\om)\ZZ[\om]$
and thus $u$ lies
on a reflecting hyperplane
of~(\ref{somehyperplanes}).
If $\beta$ itself lies
in $(1-\om)\ZZ[\om]$,
then so does $-\om\beta$,
so we assume $\beta\notin \ZZ[\om]$.
Then $2\beta\notin\ZZ[\om]$
as well, since
$(1-\om)\ZZ[\om]=(1-\om^2)\ZZ[\om]=\{a+b\omega:\,a+b\equiv 0 \mod 3\}$. 
This implies that the complex
numbers 
$\alpha$, $\alpha+\beta$,
and $\alpha+2\beta$ lie
in different cosets of the subgroup $(1-\om)\ZZ[\om]$
of the additive group
$\ZZ[\om]$.
As this subgroup
has index $3$ in the larger group,
one of these cosets is trivial, and thus 
$\alpha$, $\alpha+\beta$,
or $\alpha+2\beta$ 
lies in $(1-\om)\ZZ[\om]$.
But $(1-\om)\ZZ[\om]$
is closed under multiplication
by $\om$, hence 
$\alpha+2\beta$, $-\om^2(\alpha+\beta)$,
or $-\om\beta$ must 
lie in $(1-\om)\ZZ[\om]$
as well.
\end{proof}

%%%%%%%%%%%%%%%%%%%%%%%%%%%%%%%%%%%%%%%%%%%%%
\begin{prop}
  Suppose $W=G\ltimes \Lambda$ is a
  crystallographic reflection group
  acting on\/ $\CC^n$ for $n\geq 1$ 
  with $G$ equal to
  $G(4,2,n)$, $G(6,2,n)$, or $G(6,3,n)$
  but $W\neq[G(6,3,2)]_2$.
  Then $W$ has the Steinberg property.
\end{prop}
%%%%%%%%%%%%%%%%%%%%%%%%5
\begin{proof}
The claim follows for
$n=1$, see Subsection~\ref{rank-one},
so we assume $n>1$.
Note that the
affine reflecting hyperplanes
for groups $[G(4,2,n)]_2$ and $[G(4,1,n)]_1$ 
coincide and
$[G(4,2,n)]_2\subset [G(4,1,n)]_1$ 
(see~\cite{Malle}).
Similarly, the affine reflecting  hyperplanes for
$[G(4,2,n)]_1$
and $[G(4,1,n)]_2$ 
also coincide and
$[G(4,2,n)]_1\subset [G(4,1,n)]_2$.
The groups
$[G(4,2,n)]_1$ and $[G(4,2,n)]_2$ 
thus inherit the Steinberg property
from $[G(4,1,n)]_1$ and $G[(4,1,n)]_2$,
which have the property by
Proposition~\ref{[G(r,1,n)]_1}. 
(Note that $[G(4,1,n)]_1$ and $[G(4,1,n)]_2$ do not have the same set of hyperplanes.)
The group $W=[G(6,2,2)]_2$ has the Steinberg property by 
Proposition~\ref{[G(6,2,2)]_2}.

Fix
some non-reflection $g\neq 1_W$ in $W$ with nontrivial
fixed point space $V^g$
and note $\Lin(g)$ is also
not a reflection
by Lemma~\ref{geometrylemma}(3).

First suppose $W=[G(4,2,2)]_3$.
Then 
$\Lin(g^m)=\diag{-1,-1}$
for $m=1$ or $2$ and
$$\Tran(g^m)=\alpha\begin{pmatrix} i\\ -1 \end{pmatrix}+\beta(1+i)\begin{pmatrix} 1\\ -1 \end{pmatrix}$$ 
for some $\alpha,\beta\in \ZZ[i]$.  
If $g$ fixes $u$ in $V$, then 
$x_1(u)+ix_2(u)=\beta$ 
and $u$ is fixed by the affine
reflection
$s$ with $\Lin(s)$ sending
$e_1$ to $-ie_2$
and $e_2$ to $ie_1$
and $\Tran(s)=\beta e_1-i\beta e_2$.
The reflection $s$ lies in $W$
since
$\Tran(s)=(-\beta)(i, -1)
+ \beta(1+i)(1, -1)$
lies in $\Lambda$.
Thus
$[G(4,2,2)]_3$ has the
Steinberg property.

Now suppose $W=[G(6,p,n)]_1=G(6,p,n)\ltimes\Lambda$
for some lattice $\Lambda$ in 
Table~\ref{StateTable}. 
First assume $\Lin(g)$ is diagonal.
If $g^p\neq 1$, then Lemma~\ref{TechnicalLemma}(1)
applies, so we assume $g^p= 1$.
Then $\Lin(g)$ has
two nontrivial diagonal entries
$\lambda_j$ and $\lambda_\ell$
which are $p$-th roots-of-unity
as $(\Lin(g))^p=\Lin(g^p)=1$.
Consider
$$
\alpha_j=\tfrac{x_j(\Tran(g))}
{1-\lambda_j}
\quad\text{and}\quad
\alpha_\ell=\tfrac{x_\ell(\Tran(g))}
{1-\lambda_\ell}\ 
\quad\text{ in }\CC.
$$
One may check
directly (using that
$r=6$ and $p=2$ or $p=3$) that
$\alpha_j,\alpha_{\ell}$ must both lie in the set $X=\ZZ[\om]
/(1-\xi^{6/p})$. 
There are two orbits in $X$ under the action of $G(6,1,1)\ltimes\ZZ[\om]$ on $\CC$
with $\ZZ[\om]$
one of the orbits.
If $\alpha_j$ or $\alpha_{\ell}$ 
lie in different orbits,
then either
$\alpha_j$ or $\alpha_{\ell}$ lies in $\ZZ[\om]=x_1(\CC e_1\cap\Lambda)$
and Lemma~\ref{TechnicalLemma}(2) applies. 
Otherwise, both
$\alpha_j$ or $\alpha_{\ell}$ 
lie in the same orbit
and Lemma~\ref{TechnicalLemma}(3) applies
with $\Lambda'=x_1(\CC(e_1-e_2)\cap\Lambda)=\ZZ[\om]$. 
We conclude that any vector fixed by $g$  lies on a reflecting hyperplane for $W$ when $\Lin(g)$ is diagonal.
When $\Lin(g)$ is not
diagonal, we use arguments
as in the proof
of Proposition~\ref{[G(r,1,n)]_1}.
(Note that we may work up to conjugation by any element
in $G(6,1,n)$ since $G(6,p,n)$ is normal in $G(6,1,n)$ and $\Tran(W)$
 is $G(6,1,n)$-invariant;
 applying $G(6,1,n)$ to any reflecting hyperplane for $W$ will produce another reflecting hyperplane for $W$ although not necessarily in the same 
 $W$-orbit.)
  
 \end{proof}

%%%%%%%%%%%%%%%%%%%%%%%%%%%%%%%%%%%%%%%%%%%%%%%%%%%%%%5
\section{Genuine groups failing Steinberg's theorem}
\label{groupsfailing}

We now determine genuine crystallographic affine reflection groups in the 4-parameter infinite family failing Steinberg's theorem.
There are also genuine crystallographic groups $W$ with $\text{Lin}(W)\neq G(r,p,n)$ which
fail to satisfy the Steinberg property,
$[G_4]_1$ for example,
  see Cote~\cite{Cote}.
We set $\omega=e^{2\pi i/3}$
in Table~\ref{counterexamples}.

%%%% Local Table Spacers
\newcommand{\Crule}{\rule[-2ex]{0ex}{5ex}}
\begin{table}[htpb]
\caption{Genuine Crystallographic  groups failing
Steinberg's theorem}
\label{counterexamples}
\centering
\begin{tabular}{p{.2\textwidth} p{.3\textwidth} p{.17\textwidth}}
\toprule
\text{\bf Group} $W$ & 
\multicolumn{2}{l}{
$g \in W$ 
\text{\bf with} $V^g\not\subset$ %$\bigcup_H H$
\text{\bf union of reflecting hyperplanes}
\rule[-1ex]{0ex}{3.5ex}
}
\\ 
\midrule
$[G(3,1,2)]_2$ & 
$g(v)= \begin{smallpmatrix} 
%\rule[-1.5ex]{0ex}{3ex}%strut
\omega& \hphantom{xx}  
\\ 
%\rule[-1.5ex]{0ex}{3ex}%strut
\hphantom{xx}  
&\omega \end{smallpmatrix}v
+\begin{smallpmatrix}  
%\rule[-1.5ex]{0ex}{4ex}%strut
\ 1/(1-\omega)\\  
%\rule[-2ex]{0ex}{3ex}%strut
-1/ (1-\omega)  \end{smallpmatrix}$
& \text{for} \ \ $v\in \CC^2$
\Crule
\\ 
$[G(3,3,3)]_1$ &
$g(v)=\begin{smallpmatrix} \omega& & \\ &\omega& \\&&\omega\end{smallpmatrix}v
+\begin{smallpmatrix} \  \ 1\\-1 \\ \ \ 0\end{smallpmatrix}$
& \text{for} \ \ $v\in \CC^3$
\Crule
\\ \hline

$[G(4,4,3)]_1$ & $g(v)=\begin{smallpmatrix} i& & \\ &-1& \\&& i\end{smallpmatrix}v 
+\begin{smallpmatrix}\ \ 1\\-1 \\ \ \ 0\end{smallpmatrix}$
& \text{for} \ \ $v\in \CC^3$
\Crule
\\ \hline
$[G(6,3,2)]_2$ & 
$g(v)= \begin{smallpmatrix} -1& \\ &-1\end{smallpmatrix}v
+\begin{smallpmatrix} 1\\
\rule[0ex]{0ex}{0ex}%strut
\omega-1 \end{smallpmatrix}$
& \text{for} \ \ $v\in \CC^2$
\Crule
\\ 
$[G(6,6,3)]_1$ & 
$g(v)= \begin{smallpmatrix} \omega& & \\ &-1& \\&&-\omega^2\end{smallpmatrix}v
+\begin{smallpmatrix}\ \ 1\\-1 \\ \ \ 0\end{smallpmatrix}$
& \text{for} \ \ $v\in \CC^3$
\Crule
\\ 
\bottomrule
\end{tabular}

\end{table}

The next proposition also
holds for low values
of $r$ and $n$ if we exclude
the group $W=[G(2,2,3)]_1^\alpha$.
Such groups have linear parts
which are 
finite Coxeter groups
and appear in Section~\ref{Coxetersection}.

%%%%%%%%%%%%%%%%%%%%%%%%%%%%%%%%%%%%%%%%%%%%
\begin{prop}\label{thosewithout}
  The groups $[G(r,r,n)]_1$ for $r\geq 3$ and $n\geq 3$,
  $[G(3,1,n)]_2$ for $n\geq 2$,
        and $[G(6,3,2)]_2$
    do not have the Steinberg property.
\end{prop}
%%%%%%%%%%%%%%%%%%%%%%%%%%%%%%%
\begin{proof}
Table~\ref{counterexamples} records elements in each group $W$ which fix a point 
in $V$ not on
a reflecting hyperplane
for $W$.
For $W=[G(3,1,2)]_2$ or $[G(6,3,2)]_2$, one may
check directly that
the fixed point space $V^g$ of the element
$g$ given in the table does not lie on an reflecting hyperplanes for $W$.
For $W=[G(r,r,3)]_1$,
one may check that
$V^g=\{u\}$ 
for a vector $u$ in $V=\CC^3$
with 
$x_1(u)$, $x_2(u),$
and $x_3(u)$ lying
in different orbits
under the action of
$G(r,1,1)\ltimes \ZZ[\xi]$.
But each 
reflecting hyperplane for $[G(r,r,3)]_1$ is the zero set of a polynomial $x_j=\xi^m x_k+\beta$ for $m\in\BBZ$ and $\beta\in\BBZ[\xi]$,
and thus any point
on a reflecting hyperplane has two
coordinates
lying in the same
orbit of
$G(r,1,1)\ltimes \ZZ[\xi]$
acting on $\CC$.
Thus the claim follows for $n=3$,
and for larger $n$ by extending
the given $g$ by an identity transformation to act on $\CC^n$.
\end{proof}

The results in this section and last together
give our main result, restated from the introduction:

%%%%%%%%%%%%%%%%%%
\begin{theorem}
Suppose $W=G(r,p,n)\ltimes\Lambda$
is a genuine 
crystallographic complex reflection group
acting on\/ $\mathbb{C}^n$ for
some lattice $\Lambda$
with $r,p,n\geq 1$.
Then $W$ has the Steinberg property
if and only if\/ %$W=[G(3,1,n)]_1$,
%$W=[G(6,3,n)]_1$, or %$\Lin(W)$ is 
%$G(r,1,1)$, $G(4,1,n)$, %$G(6,1,n)$, $G(4,2,n)$, %or $G(6,2,n)$.
%%
%%
%equivalently, 
$W$ is 
$[G(r,1,1)]_1$,
  $[G(3,1,n)]_{1}$, $[G(4,1,n)]_{1,2}$, $[G(6,1,n)]_1$,
  $[G(4,2,n)]_{1,2}$, $[G(4,2,2)]_3$,
  $[G(6,2,n)]_1$, $[G(6,2,2)]_2$, or $[G(6,3,n)]_{1}$.
 \end{theorem}

%%%%%%%%%%%%%%%%%%%%%%%%%%%%%%%%%%%%%%%%%
%%%%%%%%%%%%%%%%%%%%%%%%%%%%%%%%%%%%%%%%%
%%%%%%%%%%%%%%%%%%%%%%%%%%%%%%%%%%%%%%%%%
\section{Complexification of
finite Coxeter groups}
\label{Coxetersection}
Now for the nongenuine case.
We consider
crystallographic
reflection groups
$G(r,p,n)\ltimes \Lambda$
acting on $V=\CC^n$
where $G(r,p,n)$
 is the complexification of a finite Coxeter group
 acting on $\RR^n$.
We exclude complexifications of affine Weyl groups, as they are not crystallographic,
although the groups $G(r,p,n)$
that give rise to nongenuine crystallographic
groups
are all Weyl groups.
The explicit lattices
 $\Lambda$ appear
 in Table~\ref{CoxeterTable}
 with,
again, $\xi$ a primitive $r$-th root-of-unity in $\CC$.
We include groups whose
linear part is the 
Weyl group $W(A_{n-1})$,
the irreducible reflection representation of $G(1,1,n)\cong \mathfrak{S}_n$.
We also include groups with linear part $G(r,r,2)$
for $r>2$ since
$G(r,r,2)$ is the complexification of the dihedral groups $D_{2r}$
of order $2r$
after change-of-basis.
We use Popov's
notation with
each $\alpha$
a complex parameter in the {\em modular strip}
$$
\Omega=
\{z \in \CC:
-(1/2)\leq \Re(z) < 1/2, 
1\leq |z| \text{ for }
\Re(z) \leq 0,
1< |z| \text{ for }
\Re(z) > 0\}.
$$

Note that up to equivalence, there are only 3 parameters 
of crystallographic groups with linear part the Weyl group $W(B_2)=G(2,1,2)$
and Popov choose to label
these $[G(2,1,2)]^\alpha_k$
for $k=1,2,3$ and $\alpha$ in $\Omega$.
As this notation conflicts with his notation for higher dimensional groups, we use indices $k=1,2,4$ instead so that the lattices for $[G(2,1,2)]^\alpha_k$ and
$[G(2,1,n)]^\alpha_k$ for $n>2$
are always analogous.
Note that the group 
$[G(2,1,2)]^\alpha_3$
is equivalent to
$[G(2,1,2)]^\alpha_4$
and 
$[G(2,1,2)]^\alpha_5$
is equivalent
to $G[(2,1,2)]^\alpha_1$,
thus we restrict
to
$n\geq 3$
for $[G(2,1,n)]^\alpha_5$ 
in Table~\ref{CoxeterTable}.

% LOCAL SPACERS
\newcommand{\strutF}{\rule[-1.5ex]{0ex}{4ex}}
%%%%%%%%%%%%%%%%%%%%
\begin{small}
\begin{table}[h!tbp] % h! = put it here
\caption{ Crystallographic
    reflection groups
    $W=\text{Coxeter Group}\ltimes\Lambda$}
\label{CoxeterTable}
  \centering
\begin{tabular}{p{.15\textwidth}p{.1\textwidth}p{.44\textwidth}p{.15\textwidth}
}
  \toprule
{\bf Group $W$} & {\bf dim} &
 {\bf $G(r,p,n)$-invariant lattice $\Lambda$}
& {\bf Steinberg's thm}  
\\
\midrule 
\strutF
$[W(A_{n-1})]^{\alpha}_1$
& $n-1\geq 2$
& $\Sigma_{j=2}^{n}(\ZZ + \ZZ\alpha)(e_{j-1}-e_j)$
& \cmark 
\\ %\hline
\midrule
\strutF
$[G(2,1,1)]^\alpha_1$
& $n=1$
& $(\ZZ + \ZZ\alpha)e_1$
& \cmark 
\\\midrule
\strutF
$[G(2,1,n)]^\alpha_1$
& $n\geq 2$
& $(\ZZ + \ZZ\alpha)e_1+\Sigma_{j=2}^{n}(\ZZ + \ZZ\alpha)(e_{j-1}-e_j)$
& \cmark 
\\ %\hline 
\strutF
$[G(2,1,n)]^{\alpha}_2$
& $n\geq 2$
& $(\ZZ + \ZZ\alpha)e_1+\Sigma_{j=2}^{n}
\left(\ZZ+\ZZ\frac{1+\alpha}{2}\right)(e_{j-1}-e_j)$
& \xmark 
\\ %\hline
\strutF
$[G(2,1,n)]^{\alpha}_3$
& $n\geq 2$
& $(\ZZ + \ZZ\alpha)e_1+\Sigma_{j=2}^{n}\, \left(\frac{1}{2}\ZZ+\ZZ\alpha\right)(e_{j-1}-e_j)$
& \xmark 
\\
\strutF
$[G(2,1,n)]^{\alpha}_4$
& $n\geq 2$
& $(\ZZ + \ZZ\alpha)e_1+\Sigma_{j=2}^{n}\left(\ZZ+\ZZ\frac{\alpha}{2}\right)\ZZ(e_{j-1}-e_j)$
& \xmark 
\\ %\hline 
\strutF
$[G(2,1,n)]^{\alpha}_5$
& $n\geq 3$
& $(\ZZ + \ZZ\alpha)e_1+\Sigma_{j=2}^{n}
(1/2)(\ZZ + \ZZ\alpha)(e_{j-1}-e_j)$
& \xmark 
\\ \midrule
\strutF
$[G(2,2,3)]^{\alpha}_1$
& $n=3$
& $(\ZZ + \ZZ\alpha)(-e_1-e_2)+\Sigma_{j=2}^{3}(\ZZ + \ZZ\alpha)(e_{j-1}-e_j)$
& \cmark 
\\ %\hline 
\strutF
$[G(2,2,n)]^{\alpha}_1$
& $n\geq 4$
& $(\ZZ + \ZZ\alpha)(-e_1-e_2)+\Sigma_{j=2}^{n}(\ZZ + \ZZ\alpha)(e_{j-1}-e_j)$
& \xmark 
\\ \midrule
\strutF
$[G(6,6,2)]^{\alpha}_1$
& $n= 2$
& $(\ZZ + \ZZ\alpha)(\xi e_1-e_2)+(\ZZ + \ZZ\alpha)(1-\xi^2)(e_1-e_2)$
& \xmark 
\\ %\hline 
\strutF
$[G(6,6,2)]^{\alpha}_2$
& $n= 2$
& $(\ZZ + \ZZ\alpha)(\xi e_1-e_2)+\left(\ZZ+\ZZ\frac{\alpha}{3}\right)(1-\xi^2)(e_1-e_2)$
& \xmark 
\\ %\hline 
\strutB
$[G(6,6,2)]^{\alpha}_3$
& $n= 2$
&$(\ZZ + \ZZ\alpha)(\xi e_1-e_2)+\left(\ZZ+\ZZ\frac{1+\alpha}{3}\right)
(1-\xi^2)(e_1-e_2)$
& \xmark 
\\ %\hline 
\strutF
$[G(6,6,2)]^{\alpha}_4$
& $n= 2$
&  $(\ZZ + \ZZ\alpha)(\xi e_1-e_2)+\left(\ZZ+\ZZ\frac{2+\alpha}{3}\right)(1-\xi^2)(e_1-e_2)$
& \xmark 
\\ \bottomrule
\end{tabular}
\end{table} 
\end{small}

%\clearpage 
%%%  Print all floating tables already!
%%%%%%%%%%%%%%%%%%%%%%%%%%%%%%%%%%%%%%%%%%%%%%%%%%%%%%%%%%%%%%%%%%%%%%%%%%%%%%%%%%%%%
\begin{theorem}
\label{nongenuine}
  Let $W=G(r,p,n)\ltimes \Lambda$ be a crystallographic reflection group whose
  linear part $\Lin(W)$ 
  is the complexification of
  a finite Coxeter group.
    Then $W$ has the 
  Steinberg property if and only if\/
  $W$ is $[W(A_{n-1})]^{\alpha}_1$, $[G(2,1,n)]^{\alpha}_1$, or $[G(2,2,3)]_1$.
   \end{theorem}
 %%%%%%%%%%%%%%%%%%%%%%%%%%%%%%%%%%%%%%%%%%%%%%%%%
\begin{proof}
We assume $n>1$,
as the claim holds for
$n=1$ (see Subsection~\ref{rank-one}).
The groups
$[W(A_{n-1})]_1^\alpha$
have the Steinberg property by Proposition~\ref{symmetric}.
If $W=[G(2,1,n)]_1^\alpha$,
arguments as in the proof of Proposition~\ref{[G(r,1,n)]_1} show that $W$ has the Steinberg property.

Now suppose $W=[G(2,2,3)]^{\alpha}_1$
and consider
a nonidentity element
$g\in W$ fixing a point in $V$. We assume $g$ itself is not a reflection
and thus $\Lin(g)$
is also not a reflection
by Lemma~\ref{geometrylemma}(3).  Then some power
of $\Lin(g)$ must be conjugate 
by an element of
$G(2,2,3)$ to $\diag{-1,-1,1}$
or to a
$3$-cycle in
$\mathfrak{S}_3$.
If conjugate to a $3$-cycle, 
we appeal to 
Lemma~\ref{cyclecase}.
Thus we assume $\Lin(g)=\diag{-1,-1,1}$. Notice
$$
\Tran(g)= \beta\begin{pmatrix}-1\\-1\\0\end{pmatrix}+\gamma\begin{pmatrix}1\\-1\\0\end{pmatrix}+\delta\begin{pmatrix}0\\1\\-1\end{pmatrix}
\quad\text{ for some }\ \  \beta,\gamma,\delta\in(\ZZ + \ZZ\alpha).
$$  
By Lemma~\ref{geometrylemma}(1), $g$ has finite order, so $\delta=0$. If $g(u)=u$, then $x_1(u)=\frac{\gamma-\beta}{2}$ and $x_2(u)=\frac{-\gamma-\beta}{2}$ and $u$ lies on the reflecting hyperplane $\ker(x_1+x_2+\beta)$ of $W$. (This hyperplane is fixed by the reflection in $W$ with linear part sending $e_1$ to $-e_2$, $e_2$ to $-e_1$,
and $e_3$ to $e_3$
and translational part $(\beta,\beta,0)\in\Lambda$.)

The remaining nongenuine crystallographic complex reflection groups do not have the Steinberg property.
For the groups $W=[G(2,1,n)]^\alpha_k$
with $k\neq 1$ and
$\alpha \in \Omega$, observe that each reflecting hyperplane
has the form
$H=\ker(x_j-\beta)$ for some $\beta\in (1/2)(\ZZ + \ZZ\alpha)$ or the form $H=\ker(x_j\pm x_{\ell}-\beta)$ for some $\beta$ in $x_1(\Lambda\cap \CC (e_1-e_2))$. The reflecting
hyperplanes for $[G(2,2,n)]^{\alpha}_1$ all have the form $H=\ker(x_j\pm   x_{\ell}-\beta)$ for some $\beta\in (\ZZ + \ZZ\alpha)$. It is straightforward to check that the group elements $g$ in Table~\ref{counterexamplesCoxeter}
each fix a point that is not on one of these hyperplanes, after extending $g$ by the identity to define a transformation on $\CC^n$. For the groups $W=[G(6,6,2)]_k^\alpha$, 
one can similarly
check directly that the fixed point set of the element $g$ in $W$ in Table~\ref{counterexamplesCoxeter} does not lie on any reflecting hyperplane.
\end{proof}

%%%%%%%% LOCAL TABLE SPACERS
\newcommand{\Drule}{\rule[-2ex]{0ex}{4.5ex}}%row height
\newcommand{\strutD}{\rule[0ex]{0ex}{0ex}}% inside matrices
\newcommand{\strutE}{\rule[0ex]{0ex}{0ex}}% inside little matrices
%%%%%%%%%%%%%5

\begin{table}[h!tpb]
\caption{Crystallographic  groups failing
Steinberg's theorem}
\label{counterexamplesCoxeter}
\centering
\begin{tabular}{p{.2\textwidth} p{.35\textwidth} p{.17\textwidth}}
\toprule
\text{\bf Group} $W$ & 
\multicolumn{2}{l}{
$g \in W$ 
\text{\bf with} $V^g\not\subset$
% $\bigcup_H H$,
\text{\bf union of reflecting hyperplanes}
\rule[-1ex]{0ex}{3ex}
}
\\ 
\midrule
$[G(2,1,2)]^{\alpha}_2$ &
$g(v)=\begin{smallpmatrix} -1& & \\ &-1& \strutD  \end{smallpmatrix}v
+\begin{smallpmatrix} \  \
(3+\alpha)/ 2\\-(1+\alpha)/ 2
%\frac{3+\alpha}{2}\\-\frac{1+\alpha}{2}
\strutD 
\end{smallpmatrix}$
& \text{for} \ \ $v\in \CC^2$
\Drule
\\ 
$[G(2,1,3)]^{\alpha}_3$ & $g(v)=\begin{smallpmatrix} -1& & \\ &-1& \strutD \end{smallpmatrix}v 
+\begin{smallpmatrix} \, (2\alpha+1)/ 2
%\frac{2\alpha+1}{2}
\\
-1/ 2\strutD
 \end{smallpmatrix}$
%-\frac{1}{2}\strutD \\ \ \,
& \text{for} \ \ $v\in \CC^3$
\Drule
\\
$[G(2,1,2)]^{\alpha}_4 $  & 
$g(v)=\begin{smallpmatrix} -1& & \\ &-1& \strutD \end{smallpmatrix}v 
+\begin{smallpmatrix} \, (2+\alpha)/ 2
\\-\alpha/ 2
\strutD
\end{smallpmatrix}$
& \text{for} \ \ $v\in \CC^2$
\Drule
\\ 
$[G(2,1,3)]^{\alpha}_5$ & 
$g(v)=\begin{smallpmatrix} -1& & \\ &-1& \strutD \\&& -1\end{smallpmatrix}v 
+\begin{smallpmatrix} \, (3+\alpha)/ 2\\-\alpha/ 2\strutD \\ \ \, -1/ 2\end{smallpmatrix}$
& \text{for} \ \ $v\in \CC^3$
\Drule
\rule[-3ex]{0ex}{1ex}
\\ \hline
\rule[-2.5ex]{0ex}{6.5ex}
$[G(2,2,4)]^{\alpha}_1$ & 
$g(v)= \begin{smallpmatrix} -1&&& \\ &-1&&
%\rule[-.25ex]{0ex}{1.5ex}
\\
&&-1&%\rule[-.25ex]{0ex}{1.5ex}
\\
&&&-1\end{smallpmatrix}v
+\begin{smallpmatrix} 1\\1+\alpha
%\rule[-.25ex]{0ex}{1.5ex}
\\
-\alpha
%\rule[-.25ex]{0ex}{1.5ex}
\\
 0
 \end{smallpmatrix}$
& \text{for} \ \ $v\in \CC^4$
\Drule
\\ 
\midrule
$[G(6,6,2)]^{\alpha}_1$ & 
$g(v)= \begin{smallpmatrix} 
-1& 
\rule[-.5ex]{0ex}{.5ex}\\ 
 &-1 \end{smallpmatrix}v
+\begin{smallpmatrix}  
\ \, \xi+\alpha(1+\xi)
\rule[-.5ex]{0ex}{.5ex}\\  
-1-\alpha(1+\xi)  \end{smallpmatrix}$
& \text{for} \ \ $v\in \CC^2$
\Drule
\\ 
$[G(6,6,2)]^{\alpha}_2$ & 
$g(v)= \begin{smallpmatrix} 
-1& 
\strutE \\ 
 &-1 \end{smallpmatrix}v
+\begin{smallpmatrix}  
\ \, \xi+\alpha(1+\xi)/ 3
\strutE\\  
-1-\alpha(1+\xi)/ 3  \end{smallpmatrix}$
& \text{for} \ \ $v\in \CC^2$
\Drule
\\ 
$[G(6,6,2)]^{\alpha}_3$ & 
$g(v)= \begin{smallpmatrix} 
-1& 
\strutE \\ 
 &-1 \end{smallpmatrix}v
+\begin{smallpmatrix}  
\ \, \xi+(1+\alpha)(1+\xi)/ 3
\strutE\\  
-1-(1+\alpha)(1+\xi)/ 3  \end{smallpmatrix}$
& \text{for} \ \ $v\in \CC^2$
\Drule
\\ 
$[G(6,6,2)]^{\alpha}_4$ & 
$g(v)= \begin{smallpmatrix} 
-1& 
\strutE \\ 
 &-1 \end{smallpmatrix}v
+\begin{smallpmatrix}  
\ \, \xi+(2+\alpha)(1+\xi)/ 3
\strutE\\  
-1-(2+\alpha)(1+\xi)/ 3  \end{smallpmatrix}$
& \text{for} \ \ $v\in \CC^2$
\Drule
%\rule[-3ex]{0ex}{2ex}
\\ 
\bottomrule
\end{tabular}
\end{table}

\FloatBarrier

%%%%%%%%%%%%%%%%%%%%%%%%%%%%%%%%%%%%%%%%%%%%%%%%%%%%%
%%%%%%%%%%%%%%%%%%%%%%%%%%%%%%%%
% Latex: now print any floating tables that have not appeared yet!!!
%\clearpage 
%%%%%%%%%%%%%%%%%%%%%
%%%%%%%%%%%%%%%%%%%%%%%%%%%%%%%%%%%%%%%%%%%%%%%%%%%%%

\section*{Acknowledgments}
The second author thanks Cathy Kriloff for lively and helpful discussions.

%{\small
%\bibliography{bibliography} }%
\bibliographystyle{abbrv}

%%%%%%%%%%%%%%%%%%%%%%%%%%%%%%%%%%%%%%%%%%%%%%%%%5
%%%%%%%%%%%%%%%%%%%%%%%%%%%%%%%%%%%%%%%%%%%%%%%%%5
%%%%%%%%%%%%%%%%%%%%%%%%%%%%%%%%%%%%%%%%%%%%%%%%%5

%\vspace{-2ex}

\end{document}